\documentclass[10pt]{amsart}

\usepackage{amsmath}
\usepackage{graphicx}
\usepackage{amssymb}
\usepackage{epstopdf}
\usepackage[T1]{fontenc}
\usepackage{tikz}

\newcommand{\N}{\mathbb{N}}  
\newcommand{\R}{\mathbb{R}}  

\theoremstyle{plain} 
\newtheorem{thm}{Theorem}[section]
\newtheorem{lem}{Lemma}[section]
\newtheorem{pro}{Proposition}[section]
\newtheorem{cor}{Corollary}[section]

\newtheorem{claim}{Claim}[section]
\newtheorem{qu}{Question}

\theoremstyle{definition} 
\newtheorem{defn}{Definition}[section]

\theoremstyle{remark}

\title{Unwinding spirals}
\author{Alexander Fish and Laurentiu Paunescu}

\address{School of Mathematics and Statistics, University of Sydney, Australia}
\curraddr{}
\email{alexander.fish@sydney.edu.au}
\email{laurentiu.paunescu@sydney.edu.au}
\thanks{}



\date{28 February 2016}                                           

\begin{document}
\begin{abstract}
We show that there is no  bi-Lipschitz homeomorphism of $\R^2$ that maps a spiral with a sub-exponential 
decay of winding radii to an unwinded arc. This result is sharp as shows an example of a logarithmic spiral.
\end{abstract}

\maketitle
\section{introduction}

In this paper we consider a natural question whether it is possible to find a bi-Lipschitz homeomorphism from $\R^2$ to $\R^2$ that will map a given continuous curve $C_1$ onto another continuous curve $C_2$. Recall that $h: \R^d \to \R^d$ is called a bi-Lipschitz map with constant $L > 0$, if it satisfies for all $x,y \in \R^d$:
\[
\frac{1}{L} \| x - y \| \leq \| h(x) - h(y) \| \leq L \| x-y \|.
\]

 Let  $\phi:[0,\infty) \to (0,1]$ be a continuous function  monotonically decreasing to zero. Then we correspond to $\phi$ the spiral $C_{\phi}$: 
\[
C_{\phi}(t)  = \phi(t) e^{it}, \mbox{ } t \in [0,\infty],
\]
where $C_{\phi}(\infty) = \{(0,0)\}$, see Figure \ref{spiral}. 
\bigskip

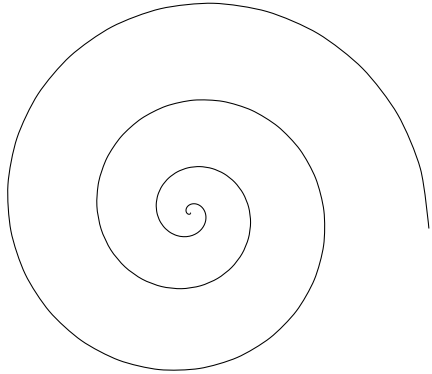
\begin{figure}[htb]
\label{spiral}
\begin{tikzpicture}
    \draw [domain=0:25.2,variable=\t,smooth,samples=100]
        plot ({-\t r}: {0.005*\t*\t});
\end{tikzpicture}
\caption{A spiral.}
\end{figure}
\bigskip

$C_{\phi}$ is  the \textit{logarithmic spiral} if $\phi(t) = e^{-at}$, for $ a > 0$. 
We will define the notion of an unwinded curve around zero in $\R^2$, after introducing the concept of asymptotic directions for a set 
$A \in \R^d$ around a point $o \in \overline{A}$. 
\smallskip

\begin{defn}[asymptotic directions at $o$]
The set of \textit{asymptotic directions} at $o$ in $A$ is defined by
\[
D(A) = \left\{ u \in S^{d-1} \, | \, \exists (a_n) \in A \mbox{ with }a_n \to o\mbox{, and } \frac{a_n - o}{\| a_n - o\|} \to u\right\}. 
\]
\end{defn} 
\smallskip

\begin{defn}[unwinded curve]
We will say that a curve $C \subset \R^2$ containing zero is \textit{unwinded} if $D(C) \neq S^1$. 
\end{defn} 
\smallskip

In this paper we will address the question of the existence of a bi-Lipschitz homeomorphism of $\R^2$ that maps the spiral $C_{\phi}$ of a finite length onto an unwinded curve. Spirals and spiral maps are very important object in geometry and have been used extensively to provide counterexamples to many natural conjectures, see for instance \cite{G}, \cite{FH}. It is well known fact that it is possible to 'unwind'  the logarithmic spiral, see \cite{KNS}:
\smallskip


\begin{thm}[unwinding of the logarithmic spiral]
There exists a bi-Lipchitz homeomorphism of $\R^2$ which maps the logarithmic spiral onto $\{(t,0) \, | \, 0 \leq t \leq 1\}$. 
\end{thm} 
\smallskip

However, it seems to be unknown, whether it is possible to 'unwind' a spiral with a sub-exponential decay of the radius-vector. We remind the notion of the sub-exponential decay. 

\begin{defn}[sub-exponential decay]
Let $\phi:[0,\infty) \to (0,1]$ be a function monotonically decreasing to zero. We  say that $\phi$  \textit{decays sub-exponentially fast} if $\frac{\log{(\phi(n))}}{n} \to 0$ as $n \to \infty$. 
\end{defn}
Notice that the rate of decay of the function $\phi$ corresponding to the logarithmic spiral is exponential. Our main result is the following:
\bigskip

\begin{thm}\label{main_thm}
Let $\phi:[0,\infty) \to (0,1]$ be a function monotonically sub-exponentially decaying to zero. There is no  bi-Lipschitz homeomorphism of $\R^2$ which maps $C_{\phi}$ into an unwinded curve. 
\end{thm}
\bigskip

Let $C_{\phi}$ be a spiral with a sub-exponential decay of the radius-vector.
It is not hard to see by use of the 'length' of the curves and the triangle inequality, see Section \ref{section4}, that there is no bi-Lipschitz homeomorphism of $\R^2$ that sends $C_{\phi}$ to the segment $\{ (t,0) \, | \, 0 \leq t \leq 1\}$. However, the impossibility of unwinding in general case requires a substantially harder argument. We will use the notion of the cone at $o \in \overline{A} \subset \R^d$.

%
\begin{defn}[cone at $o$]
The cone at $o$ of $A$ is 
\[
LD(A) = \{t u \, | \, u \in D(A), t \geq 0 \}.
\]
\end{defn}

Theorem \ref{main_thm} follows easily from the following claim.

\begin{pro}\label{main_pro}
Let $h:\R^2 \to \R^2$ be a bi-Lipschitz homeomorphism, and let $\phi:[0,\infty) \to (0,1]$ be a function monotonically sub-exponentially decaying to zero. Then there exists a bi-Lipschitz homeomorphism $\overline{h}$ of $\R^2$ with the same Lipschitz constant as $h$, such that $\overline{h}(S^1) \subset LD(h(C_{\phi}))$. 
\end{pro}

\noindent  \textit{Proof of Theorem \ref{main_thm}.} Let us assume that $\phi:[0,\infty) \to [0,1]$ is a function of a sub-exponential decay monotonically decreasing to zero.  If there exists a bi-Lipschitz homeomorphism $h: \R^2 \to \R^2$ that sends the spiral $C_{\phi}$ to an unwinded curve then  by Proposition \ref{main_pro} there exists a bi-Lipschitz map $\overline{h}$ of $\R^2$ with $\overline{h}\left(S^1\right) \subset LD(h(C_{\phi}))$. By the assumption $h$ 'unwinds' $C_{\phi}$, therefore $LD(h(C_{\phi}))$ is a proper cone in $\R^2$. It follows from the Jordan curve theorem that the image of $S^1$ under a Lipschitz homeomorphism of $\R^2$ is not inside a proper cone. Indeed, by Jordan curve theorem the image of $S^1$ divides the plane into two connected components. Since the map is open it follows that the connected component of the image of the disk $\{(x,y) \, |  \, x^2+y^2 \leq 1\}$ is unbounded one, since it  contains the complement of the cone. But the map is a homeomorphism, and we get a contradiction. 

\qed
\bigskip

\noindent \textit{Acknowledgment.} The authors are grateful to Sheehan Olver for his help with the drawings used in this paper.
\bigskip 

\section{Proof of Proposition \ref{main_pro}}\label{section2}

Given a spiral $C_{\phi}$, we will define the sequence $(r_n) $ of \textit{winding radii} corresponding to it as follows:
\[
r_1 = \phi(0), \, r_{2} = \phi(2 \pi)\, , \ldots,\,  r_n = \phi(2 \pi (n-1)), \ldots.
\]
In Section \ref{section3} we will show:

\begin{lem}\label{lemma}
Let $(r_n)$ be a sequence of a sub-exponential decay monotonically decreasing to zero. Then there exists a subsequence of indices $(n_k)$ (of density one), such that 
\[
\frac{r_{n_k+1}}{r_{n_k}} \to 1, \mbox{ as } k \to \infty. 
\]
\end{lem}
\smallskip

In view of Lemma \ref{lemma} it is natural to define the following notion.
\begin{defn}[regular sequence]
We  say that a positive monotonically decreasing to zero sequence $(r_n)$  is \textit{regular} if there is a subsequence of indices $(n_k)$ such that 
$(r_{n_k})$ satisfies the property:
\[
\lim_{k \to \infty} \frac{r_{n_k+1}}{r_{n_k}} = 1.
\]
\end{defn}

\noindent Next, we define a key property used in this paper.

\begin{defn}[$\widetilde{SSP}$ condition]
We say that $o \in \overline{A}$ satisfies a \textit{sub-sequence selection property} ($\widetilde{SSP}$) if the following holds:
\smallskip

\begin{itemize}
\item $\exists (r_n)$  regular  sequence of radii around $o$.\\
\item For every $u \in D(A)$ there exists a sequence $(a_n) \subset A$ such that $r_{n+1} \leq \| a_n - o \| \leq r_n$ for $n$ large enough, and $\frac{a_n - o}{\|a_n - o\|} \to u$ as $n \to \infty$.  
\end{itemize}
\end{defn}
\smallskip

\noindent As an immediate consequence of Lemma \ref{lemma} we obtain the following implication.
\begin{cor}\label{cor}
Let $\phi:[0,\infty) \to (0,1]$ be a function monotonically sub-exponentially decreasing   to zero. Let $(r_n)$ be the winding radii corresponding to the spiral $C_{\phi}$. Then $o = \{(0,0)\} \in \overline{C_{\phi}}$ satisfies the
($\widetilde{SSP}$) condition.
 \end{cor}
 \smallskip

\noindent Proposition \ref{main_pro}  follows immediately from the following claim and Corollary \ref{cor}.
\begin{pro}\label{prop1}
Let $A \subset \R^d$ and let $h: \R^d \to \R^d$ be a  bi-Lipschitz  homeomorphism which satisfies $h(o) = o$. Assume  that $o \in \overline{A}$ satisfies the condition  ($\widetilde{SSP}$). Then there exists a bi-Lipschitz  homeomorphism
$\overline{h}$ of $\R^d$ such that  
\[
\overline{h}\left( D(A) \right) \subset LD\left( h(A) \right).
\]
\end{pro} 
\smallskip

\noindent \textit{Proof of Proposition \ref{prop1}.} Let $o \in \overline{A} $ be satisfying condition ($\widetilde{SSP}$). Assume that $h: \R^d \to \R^d$ is a bi-Lipschitz  homeomorphism
with  $h(o) = o$. Without loss of generality, assume that $o = \mathbf{0}_{\R^d}$. Let $(r_n)$ be a regular sequence of radii around $o$ such that for every $u \in D(A)$ there exists a sequence 
$(a_n) \subset A$ such that $r_{n+1} \leq \| a_n \| \leq r_n$ for $n$ large enough, and $\frac{a_n}{\|a_n\|} \to u$ as $n \to \infty$.
Then there exists $(n_k)$ subsequence of indices such that 
\[
\frac{r_{n_k +1}}{r_{n_k}} \to 1 \mbox{ as } k \to \infty.
\]
Denote by $(h_k)$ a sequence of bi-Lipschitz homeomorphisms defined by $h_k(x) = \frac{1}{r_{n_k}}h(r_{n_k}x)$. By Arzela-Ascoli theorem, since all these maps have the same Lipschitz constant, there is a subsequence along which the limit exists. Denote the limit by $\overline{h}$, and without loss of generality, assume that $h_k \to \overline{h}$. 
\smallskip

\noindent Next, take any $u \in D(A)$.
Then for  $k \in \N$ large, $a_{n_k} \in A$ satisfies  
\begin{equation}\label{ineq_main}
\| a_{n_k} - r_{n_k} u \| \ll r_{n_k},
\end{equation}
by the triangle inequality and the identity 
\[
\| a_{n_k} - \| a_{n_k}\| u \| \ll \| a_{n_k}\|. 
\]

\noindent Now we apply $h$ on the inequality (\ref{ineq_main}), and obtain
\[
\left\| \frac{1}{r_{n_k}}h(a_{n_k}) - \frac{1}{r_{n_k}} h(r_{n_k} u) \right\| \to 0, \mbox{ as } k \to \infty. 
\]
Therefore, we have $\frac{1}{r_{n_k}}h(a_{n_k}) \to \overline{h}(u)$, which shows that  there exists $s > 0$ with $s \overline{h}(u) \in D(h(A))$. In other words, we have shown that  $\overline{h}(u) \subset LD(h(A))$.
\qed
\bigskip

\section{Proof of Lemma \ref{lemma}}\label{section3}
 Let $a_n = \frac{1}{r_n}$. Then $a_n$ is an increasing sequence, and  $a_n \leq g(n)$ satisfying $\frac{\log{(g(n))}}{n} \to 0$, as $n \to \infty$. It is enough to show that there exists a sparse set of indices $R \subset \N$, such that the sequence $b_n = \frac{a_{n+1}}{a_n}$ converges to one as $n \to \infty$ and $n \in \N \setminus R$. For every $m \in \N$, denote by 
 \[
 R_m = \left\{ n \in \N \, | \, \frac{a_{n+1}}{a_n} \geq 1 + \frac{1}{m}\right\}.
\]
Our first claim is the following.
\smallskip

\begin{claim}\label{claim1}
For $N$ large enough (independent of $m$) we have 
\[
\left|R_m \cap [1,N] \right| \leq 3 m \ln{(g(N))}.
\]
\end{claim}
\begin{proof}
For any fixed (large) $N$, denote by $c = |R_m \cap [1,N]|$. Since the largest possible index contained in $R_m \cap [1,N]$ is $N$, and $(a_n)$ is an increasing sequence,  we get the bound
\[
a_N \geq \left(1+\frac{1}{m}\right)^{c-n_o},
\] 
where $n_o$ is the smallest $n$ for which $a_n \geq 1$.
On other hand, we also have $a_N \leq g(N)$. The last two inequalities imply that
\[
(c - n_o )\log{ (1+\frac{1}{m})} \leq \ln{(g(N))}. 
\]
This implies  $c -n_o \leq \frac{\ln(g(N))}{\ln (1+\frac{1}{m})} \leq 2 m  \ln(g(N))$, where in the last inequality we used the easy fact that $\frac{1}{\ln(1+x)} \leq \frac{2}{ x}, \mbox{ for } 0  <  x \leq 1$.
\smallskip

\noindent Thus, we have for $N$ large enough (independent of $m$) that $c \leq 3 m \ln(g(N))$, which finishes the proof of the claim.
\end{proof}
\smallskip

Next, we define the sets $R_m' \subset \N$ of indices in the following way:
\[
R_m' = R_m \cap [f_1(m),\infty),
\]
where $f_1: \N \to \N$ is very fast increasing function satisfying that $\ln(g(n)) f_2^2(n) \ll n$, where
\[
f_2(n) = \min\{ m \in \N \, | \, f_1(m) \geq n\}.
\]
And, finally, let
\[
R = \cup_{m \geq 1} R_m'.
\]  
We will show that $R$ is a set of density zero. It is clear that $\frac{a_{n+1}}{a_n} \to 1,$ as $n \to \infty \mbox{ and }  n \in \N \setminus R$. Our final claim shows the sparsity of $R$.
\smallskip

\begin{claim}
For $N$ large enough we have $\left|R \cap [1,N]\right| \leq 3 \ln(g(N)) f_2^2(N)$.
\end{claim}

\begin{proof}
First, notice that for any given $N$, the only contributors to $R \cap [1,N]$ are from the sets $R_1',R_2',\ldots,R_{L}'$, where $L = f_2(N)$. Therefore, by use of Claim \ref{claim1}, we estimate
\[
\left|R \cap [1,N]\right| \leq \sum_{m=1}^L | R_m' \cap [1,N]| \leq \sum_{m=1}^L | R_m \cap [1,N]| 
\]
\[
\leq  \sum_{m=1}^L 3 m \ln(g(N)) \leq 3 L^2 \ln(g(N)) = 3 \ln(g(N)) f_2^2(N),
\]
for $N$ large enough.
\end{proof}
\bigskip

\section{Unwinding spirals to straight segments}\label{section4}

We will give an elementary argument for the following special case of Theorem \ref{main_thm}.

\begin{pro}\label{easy_pro}
Let $\phi:[0,\infty) \to (0,1]$ be a function monotonically sub-exponentially decaying to zero. There is no  bi-Lipschitz homeomorphism of $\R^2$ which maps $C_{\phi}$ into the line segment $\{(t,0) \, | \, 0 \leq t \leq 1 \}$.
\end{pro}

We will use the notion of the length for a curve. Let $h:\R^2 \to \R^2$ be a bi-Lipschitz map. Let $I=[a,b]$ be a subinterval of $\{(t,0) \, | \, 0 \leq t \leq 1 \}$. By a finite partition $P$ of $I$ we mean $t_0 = a < t_1 < \ldots < t_k = b$. To the  partition $P$ we correspond 
the approximate length of $h(I)$ along $P$ as follows:
\[
\mathcal{L}_P(h(I)) = \sum_{j=1}^k \| h(t_j) - h(t_{j-1}). \|
\]
Then  the length of $h(I)$ is defined by 
\[
\mathcal{L}(h(I)) = \sup_{P \mbox{ is a finite partition of } I} \mathcal{L}_P(h(I)).
\]
Notice that if $h(I)$ is a rectifiable curve, then the usual length coincides with the one that we just defined.
If $h$ has Lipschitz constant $L$, then we have $\frac{|I|}{L} \leq \mathcal{L}(h(I)) \leq L |I|$, where $|I|$ denotes the length of $I$ equal to $b-a$. 
\bigskip

\noindent \textit{Proof of Proposition \ref{easy_pro}.} Let us assume that there exists a bi-Lipschitz homeomorphism $h$ of $\R^2$ with Lipschitz constant $L$ that sends a spiral $C_{\phi}$, satisfying the sub-exponential decay condition on winding radii, to the line segment $\{(t,0) \, | \, 0 \leq t \leq 1 \}$. Without loss of generality, we can assume that $h(C_{\phi}) = \{(t,0) \, | \, 0 \leq t \leq 1 \}$. Then by Lemma \ref{lemma} from the sequence of winding radii $(r_n)$ corresponding to the spiral $C_{\phi}$ we can extract a subsequence $(r_{n_k})$ such that 
\begin{equation}\label{fignja}
\frac{r_{n_k+1}}{r_{n_k}} \to 1, \mbox{ as } k \to \infty. 
\end{equation}
Then we look at the closed curve $\Gamma_k$ in $\R^2$ comprising the part of the spiral $\Gamma_k'$ between the radii $r_{n_k}$ and $r_{n_{k}+1}$, and the line segment $I_k$ connecting $r_{n_k+1}$ and $r_{n_k}$, see Figure \ref{fig2}. 
\bigskip

\begin{figure}[htb]
\begin{center}
\includegraphics[clip,height=5cm]{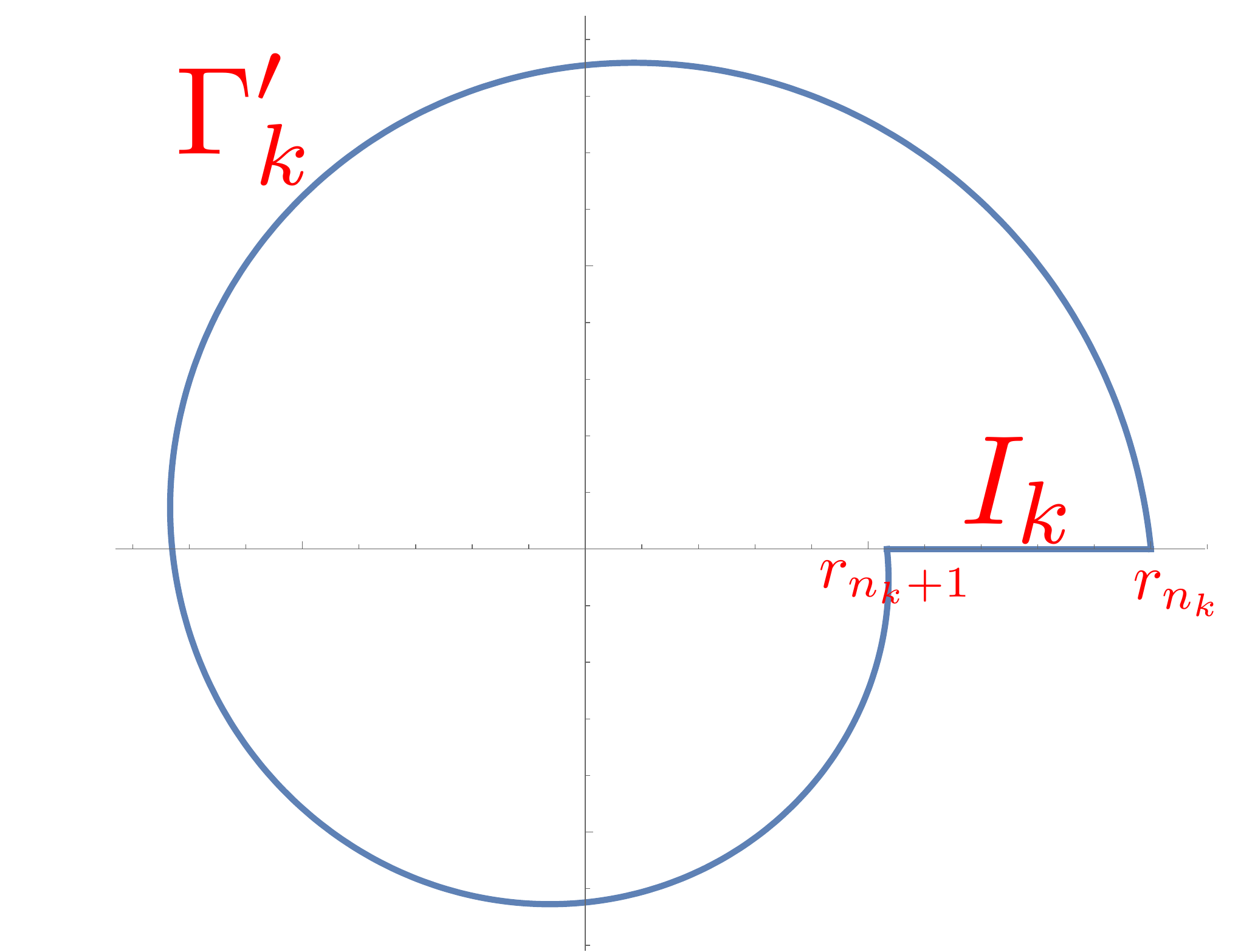}\\
\caption{The curve $\Gamma_k = \Gamma_k' \cup I_k$.}
\label{fig2}
\end{center}
\end{figure}
\bigskip

Since the length of $\Gamma_k'$ is greater or equal than $r_{n_{k}+1}$, it follows that $|h(\Gamma_k')| \geq \frac{r_{n_k+1}}{L}$. On the other hand, $|I_k| = r_{n_k} - r_{n_k+1} \ll r_{n_k+1}$ by (\ref{fignja}). Therefore $\mathcal{L}(h(I_k)) \leq L | I_k | \ll r_{n_k+1}$. Notice that  $h(\Gamma_k')$ is mapped into the straight segment which is the geodesic connecting the points $h((r_{n_k},0))$ and $h((r_{n_k+1},0))$ in $\R^2$. The endpoints of $h(I_k)$ coincide with the endpoints of $h(\Gamma_k')$, and therefore we must have 
\[
\mathcal{L}(h(I_k)) \geq |h(\Gamma_k')|,
\]
which is impossible. This finishes the proof of the proposition.

\qed

\bigskip

\section{Appendix -- Remarks on the (SSP) condition}

The sequence selection property (SSP) condition for the set of directions at a point $o \in \overline{A} \subset \R^d$ was introduced by Koike and the second author in \cite{KP1}, and has been studied in \cite{KP1}, \cite{KP2} and \cite{KP3}. 

\begin{defn}[(SSP) condition]
The set of directions at $o \in \overline{A} \subset \R^d$ satisfies (SSP) condition if for every $u \in D(A)$ and every positive sequence $(t_n)$ decreasing to zero, there exists a sequence $(a_n) \in A$ such that 
\[
\| a_n - t_n u \| \ll \max{(\|a_n\|, t_n)}.
\]
\end{defn}
It is proved in \cite{KP3} the following.
\smallskip

\begin{pro}[Lemma 2.9, \cite{KP3}] \label{pro2}
Let $h: \R^d \to \R^d$ is a bi-Lipschitz homeomorphism with $h(o) = o$, for $o \in \overline{A} \subset \R^d$, and $o \in \overline{A}$ satisfies (SSP) condition. 
Then there exists a bi-Lipschitz homeomorphism $\overline{h}$ of $\R^d$
such that $\overline{h}(LD(A)) \subset LD(h(A))$. 
\end{pro}
\smallskip

The main consequence of Proposition $\ref{pro2}$ is that for any reasonable (satisfying the monotonicity property along Lipschitz maps) notion of the dimension in $\R^d$, we will have the following.

\begin{cor}
Let $h:\R^d \to \R^d$ be a bi-Lipschitz homeomorphism with $h(o) = o$, and let $A \subset \R^d$ be sets with $o \in \overline{A} $. If $o \in \overline{A}$ satisfies the (SSP) condition, then 
\[
Dim(LD(A)) \leq Dim(LD(h(A))). 
\]  
\end{cor}
\smallskip

The  sub-exponential decay condition on $\phi$ does not guarantee that $C_{\phi}$ satisfies the (SSP) condition. Notice that the conclusion in Proposition \ref{prop1} is much weaker than in Proposition \ref{pro2}.  We can obtain almost the same conclusion as in Proposition \ref{pro2}, if we will require slightly more from a regular sequence of radii $(r_n)$  appearing in the definition of the $\widetilde{SSP}$ condition.  It is not hard to prove the following statement by use of the similar techniques as in the proof of Proposition \ref{prop1}.

\begin{pro}\label{techn_pro}
Let $h: \R^d \to \R^d$ is a bi-Lipschitz homeomorphism with $h(o) = o$, for $o \in \overline{A} \subset \R^d$. Assume that  $o \in \overline{A}$ satisfies ($\widetilde{SSP}$) condition. In addition, assume that the regular family of radii $(r_n)$ has the property that there exists $M > 0$ with
\[
\sup_{n > m} \frac{r_n - r_{n+1}}{r_m - r_{m+1}} < M, \mbox{ for all } m \in \N.
\]
Then there exists a bi-Lipschitz homeomorphism $\overline{h}$ of $\R^d$
such that $\overline{h}(L_{\leq 1}D(A)) \subset LD(h(A))$, where $L_{\leq 1} D(A) = \{ tu \, | \, 0 < t \leq 1, u \in D(A)\}$.
\end{pro}
\smallskip

As a corollary of Proposition \ref{techn_pro} we obtain the following.

\begin{cor}\label{int_cor}
Let $h:\R^d \to \R^d$ be a bi-Lipschitz homeomorphism with $h(o) = o$, and let $A \subset \R^d$ be sets with $o \in \overline{A} $. If $o \in \overline{A}$ satisfies the ($\widetilde{SSP}$) condition with the sequence of radii $(r_n)$ satisfying
\[
\sup_{n > m} \frac{r_n - r_{n+1}}{r_m - r_{m+1}} < M, \mbox{ for all } m \in \N.
\]
Then
\[
Dim(LD(A)) \leq Dim(LD(h(A))). 
\]  
\end{cor}

We finish the Appendix with the following natural question.
\bigskip

\begin{qu}
Is it true that the conclusion in Corollary \ref{int_cor} still holds true under a weaker assumption that $o \in \overline{A} $ satisfies the ($\widetilde{SSP}$) condition?
\end{qu}

  \end{document}